\documentclass[10pt]{article}

\usepackage{amsmath,amssymb, amsthm, amsfonts}

 \usepackage{amssymb}
 \usepackage{amsthm}
 \usepackage{amsmath}
 \usepackage{enumerate}
 \usepackage{amsfonts}
 \usepackage{txfonts}
\usepackage{pxfonts}
\usepackage{latexsym}
\title{Flocking particles in a non-Newtonian shear thickening fluid}
\author{Piotr B. Mucha\footnote{{\it Email address:} {\tt p.mucha@mimuw.edu.pl}}\ \ \ \ \ Jan Peszek\footnote{{\it Email address:} {\tt j.peszek@mimuw.edu.pl}}\\
{\it\footnotesize Institute of Applied Mathematics and Mechanics,}\\
{\it\footnotesize University of Warsaw,}\\
{\it\footnotesize ul. Banacha 2, 02-097 Warsaw, Poland}\\
Milan Pokorn\' y\footnote{{\it Email address:} {\tt pokorny@karlin.mff.cuni.cz}}\\
{\it\footnotesize Charles University, Faculty of Mathematics and Physics,}\\
{\it\footnotesize Sokolovsk\' a 83, 186 75 Prague, Czech Republic}
}

\date{\today}

\renewcommand{\it}{\sl}

plus 6pt minus 12pt
\newcommand{\barint}{
         \rule[.036in]{.12in}{.009in}\kern-.16in
          \displaystyle\int  }
\def\r{{\mathbb{R}}}

\begin{document}

\newtheorem{theo}{\bf Theorem}[section]
\newtheorem{coro}{\bf Corollary}[section]
\newtheorem{lem}{\bf Lemma}[section]
\newtheorem{rem}{\bf Remark}[section]
\newtheorem{defi}{\bf Definition}[section]
\newtheorem{ex}{\bf Example}[section]
\newtheorem{fact}{\bf Fact}[section]
\newtheorem{prop}{\bf Proposition}[section]
\newtheorem{prob}{\bf Problem}[section]

\makeatletter \@addtoreset{equation}{section}
\renewcommand{\theequation}{\thesection.\arabic{equation}}
\makeatother

\newcommand{\ds}{\displaystyle}
\newcommand{\ts}{\textstyle}
\newcommand{\ol}{\overline}
\newcommand{\wt}{\widetilde}
\newcommand{\ck}{{\cal K}}
\newcommand{\ve}{\varepsilon}
\newcommand{\vp}{\varphi}
\newcommand{\pa}{\partial}
\newcommand{\rp}{\mathbb{R}_+}
\newcommand{\hh}{\tilde{h}}
\newcommand{\HH}{\tilde{H}}
\newcommand{\cp}{{\rm cap}^+_M}
\newcommand{\hes}{\nablax ^{(2)}}
\newcommand{\nn}{{\cal N}}
\newcommand{\dix}{\nablax \cdot}
\newcommand{\dv}{{\rm div}_v}
\newcommand{\di}{{\rm div}_x}
\newcommand{\pxi}{\partial_{x_i}}
\newcommand{\pmi}{\partial_{m_i}}
\newcommand{\tor}{\mathbb{T}}
\newcommand{\pot}{\mathcal{v}}
\newcommand{\nablav}{\nabla_v}
\newcommand{\nablax}{\nabla_x}
\newcommand{\FCS}{F_{CS}}
\newcommand{\Fd}{F_d}

\maketitle
\begin{abstract}
\noindent
We prove existence of strong solutions to the Cucker--Smale flocking model coupled with an incompressible viscous non-Newtonian fluid, with the stress tensor of a power--law structure for $p\geq\frac{11}{5}$. The fluid part of the system admits strong solutions while the solutions to the Cucker--Smale part are weak.
The coupling is performed through a drag force on a periodic spatial domain $\tor^3$. Additionally we construct a Lyapunov functional determining large time behavior of solutions to the system. 
\end{abstract}
\noindent

%
%
%
%

\section{Introduction}

Mathematical models of self-propelled agents with non-local interactions provide a way to describe a wide range of phenomena in natural sciences: physics, biology,
but also in economics or even in robotics. The literature concentrates on  analysis of time asymptotics  
\cite{hakalaru, hamd}
, pattern formation \cite{hajekaka, top, tontu}
 and study of  models with  forces that simulate various 
natural factors \cite{car3, dua1} (deterministic case) or \cite{cuc4} (stochastic one). 
The other variations of the model include forcing particles to avoid collisions \cite{cuc2} or to aggregate under the leadership of certain individuals \cite{cuc3}. 
%

We concentrate on the  Cucker-Smale (CS) flocking model describing a collective self-driven motion of self-propelled particles with a tendency to flock. 
The system has been introduced by Cucker and Smale in \cite{cuc1} in 2007 and it initiated intensive study of the subject from the mathematical point of view.
 The vast literature on the CS model refers mostly to qualitative analysis \cite{car, top, hahaki}. Simple form of the system allows unexpectedly to find answers to questions concerning the structure of solutions like aggregation with leaders \cite{cuc3, shen}, collision avoidance \cite{ahn1, ccmp}, cluster formation \cite{ha}. The theory contains also examination of systems with various additional forces \cite{dua1, halele} and with special cases of the communication weight: singular \cite{haliu, ahn1, jpe, jps}, normalized \cite{mo} and incorporating the effect of time-delay \cite{has}. Parallely, research on the passage from the particle CS system to the kinetic equation is performed \cite{haliu,hatad,mp} (see also \cite{rec,deg1,deg2} for general theory on derivation of kinetic models from non-local particle systems). 

The present paper considers one of the other directions of research. Our subject is   motion of agents described by the kinetic CS equation 
\begin{equation} \label{X}
\partial_t f + v\cdot \nablax f+\dv (F(f)f) = 0,
\end{equation}
submerged in a non-Newtonian viscous incompressible fluid. In parallel to the analysis of the kinetic models themselves, research on coupling models of kinetic theory with models of hydrodynamics 
was conducted (see \cite{org,go2,go3}) and they are a part of large theory called complex flows.
Our motivation comes from results for the  complex flow models, here we shall mention \cite{go1,CoSe} concerning the Fokker-Planck equation coupled with the Navier-Stokes system.
The literature on the CS model coupled with models of hydrodynamics is quite rich. It includes coupled CS-Navier-Stokes system \cite{Bae}, also in the compressible case \cite{choi}, and a venture towards well-posedness with small data \cite{bae2}. 
We aim at proving global in time solvability for arbitrary large data with solutions with a regular fluid part. Note that for the classical Navier-Stokes equations we are still not able to consider general smooth solutions, hence application of the non-Newtonian concept of description of the flow allows to obtain stronger results than for the Newtonian case like in \cite{Bae} and unlike \cite{bae2} it does not require smallnes of initial data.


Our goal is to consider particles embedded in an incompressible viscous non-Newtonian shear thickening fluid, i.e. we aim to couple \eqref{X} with the system
\begin{align}\label{ns}
\left\{
\begin{array}{rcc}
\partial_t u + (u\cdot\nablax)u + \nablax\pi - \di(\tau) &=& G_{ext},\\
\di u &=& 0,
\end{array}
\right.
\end{align}
which describes the motion of such a fluid. The function
\begin{align*}
u=u(t,x)=(u_1(t,x),u_2(t,x),...,u_d(t,x))
\end{align*}
represents velocity of the fluid at the position $x$ and time $t$. Equation $(\ref{ns})_2$ expresses the conservation of mass (as well as the incompressibility constraint), while
$(\ref{ns})_1$ expresses the conservation of momentum. The term $\tau$ in $(\ref{ns})_1$ denotes a symmetric stress tensor that depends on $Du$ --- the symmetric part of the gradient of $u$,
i.e. $\tau=\tau(Du)$, where $Du=\frac{1}{2}[\nablax  u+(\nablax  u)^T]$. Function $G_{ext}$ represents an external force.

To couple \eqref{X} with (\ref{ns}) we introduce the following drag force
\begin{align*}
F_d(t,x,v):= u(t,x)-v,
\end{align*}
that influences the motion of  particles and fluid. 
 Explicitly, the coupled system reads as follows:
\begin{align}\label{sys}
\left\{
\begin{array}{rcl}
\partial_t f + v\cdot \nablax f+\dv [(\FCS(f)+\Fd)f] &=& 0,\\
\partial_t u + (u\cdot\nablax)u + \nablax\pi - \di(\tau) &=& -\int_{\r^d}\Fd fdv,\\
\di u &=& 0.
\end{array}
\right.
\end{align}
The system is considered over the phase-space $\tor^d_x \times \r^d_v$ with a set of initial data.
Our main result is presented by Theorem \ref{exiful.p2}. It says that for any given initial velocity and distribution of particles assumed to be suitably regular there exists global in time regular solution, provided the growth of
stress tensor $\tau(Du)$ is greater than $p-1$ with $p\geq \frac{11}{5}$, the same as for the pure non-Newtonian fluid \cite{mal}. In addition, we construct a Lyapunov functional which shows that the energy of the system decays in time, 
in the special case for $p>3$, in which we are able to conclude that the energy goes to zero as time goes to infinity. This type of studies of large-time behavior can be found in \cite{bae2} and in a particularly refined version in \cite{choi}.

Let us briefly discuss the difference between coupling of the CS model with Newtonian and non-Newtonian fluids. In \cite{org,Bae}, the authors obtained 
existence of weak solutions for their coupled systems and on top of that in \cite{Bae}, the authors proved asymptotic flocking (adding later in \cite{bae2} a modification of the large-time behavior part of the result in a small initial data scenario). In case of coupling with a non-Newtonian fluid, existence,
regularity and possibly uniqueness depend on the value of the exponent $p$ and regularity of the external function $G_{ext}$. For the non-Newtonian system (\ref{ns}) existence of weak 
solutions is known for $p>\frac{2d}{d+2}$ and $G_{ext}\in L^p(0,T;(W^{1,p})^*(\tor^d))$.
On the other hand, if $p\geq \frac{3d+2}{d+2}$ and 
 $G_{ext}
 \in L^2(0,T;L^2(\tor^d))$, we have not only  existence of strong solutions but also their uniqueness \cite{poko}. However for the coupled system uniqueness is a more delicate problem since $F_d$ in $\eqref{sys}_1$ forces the particles to move along trajectories influenced by $u$. It results in the need to control the $L^\infty$ norm of $\nablax u$, which in the non-Newtonian case is very difficult even with large $p$.
  
The paper is organized as follows. First we introduce the system (\ref{sys}) and formulate the main results. In Section 3, the kernel of the paper, we prove Theorem \ref{exiful.p2}. In Section 4 we deal with large-time behavior of solutions. Finally in Appendix a number of auxiliary results are presented/proved.

%
%
%
%

\section{Preliminaries}
Introduce the notation. By $W^{k,p}(\Omega)$ we denote the Sobolev space of functions with up to the $k$-th weak derivative belonging to the Lebesgue space $L^p(\Omega)$. 
Moreover, ${\mathcal D}^{'}(\Omega)$ denotes the space of distributions on $\Omega$ and $C^k(\Omega)$ ---
the space of the functions with up to the $k$-th derivative belonging to the space of continuous functions, which itself is denoted as $C(\Omega)$. 
The norm $\|\cdot\|_q$ denotes the $L^q$-norm, either over $\tor^3$ or over $\r^3\times \tor^3$, in dependence on the function which norm we have in  mind.
 In case it will be necessary to distinguish, we will use the full notation of the norm. The same holds in the case when the time variable is considered.
 We  also use 
\begin{align*}
A\stackrel{H(q)}{\leq} B
\end{align*}
to emphasize that the estimate $A\leq B$ follows by H\" older's inequality with exponent $q$. We  use a similar notation for Young's inequality replacing $H$ with $Y$.
An arbitrary generic constant is denoted by $C$; its actual value may change depending on its appearances even in the same line.

Let us specify the structure of the main system (\ref{sys}). We start with explanation for the equations on motion of non-Newtonian fluid. The sought elements are the velocity $u$ and pressure $p$
defined over the $d$-dimensional periodic box and time interval $[0,T]$. For the stress tensor $\tau:\r^{d^2}_{sym}\to\r^{d^2}_{sym}$ there exist $p\in(1,\infty)$ and positive constants $c_1 - c_5$, such that for all $\xi, \eta\in \r^{d^2}_{\rm sym}$
\begin{align}
 \tau_{ij}(\xi)\xi_{ij}&\geq c_1 (|\xi|^p + |\xi|^2), \qquad  
|\tau_{ij}(\xi)| \leq c_2(1+|\xi|)^{p-1},\label{tau2.p2}\\
(\tau_{ij}(\xi)-\tau_{ij}(\eta))(\xi-\eta)&\geq c_3(|\xi-\eta|^2 + |\xi-\eta|^p),\label{tau3.p2}\\
\frac{\partial \tau_{ij}(\eta)}{\partial \eta_{kl}}\xi_{ij}\xi_{kl} &\geq c_4 (1+|\eta|)^{p-2} |\xi|^2, \qquad  \Big|\frac{\partial \tau_{ij}(\eta)}{\partial \eta_{kl}}\Big| \leq c_5 (1+|\eta|)^{p-2}. \label{tau4.p2}
\end{align}
%
As a classical example we point out 
$ \tau(\xi)=C(1+|\xi|)^{p-2} \xi,$
keeping in mind that $\xi$ is meant as the symmetric part  of the velocity gradient,  i.e. $\xi = Du=\frac 12 (\nablax  u + (\nablax  u)^T)$.

Regarding the CS part of the system, we look for distribution function $f$ defined over the phase-space $\tor^d_x \times \r^d_v$ for $t \in [0,T]$. The function is required to be non-negative.
The equation on $f$ is coupled through the force term $F(f) = \FCS(f) + \Fd(f)$, where
\begin{equation}\label{drag}
 F_d(t,x,v):= u(t,x)-v,
\end{equation}
and 
\begin{equation}\label{FCS}
 \FCS(f)(t,x,v)=\int_{\tor^d\times\r^d} (w-v)\psi(|x-y|) f(t,y,w) dy dw,
\end{equation}
where $\psi(\cdot)$ -- the communication weights is non-negative, non-increasing and smooth, with
$\|\psi\|_{{\mathcal C}^1}\leq c_6.$ It follows 
$F_{CS}(f)(t,x,v)=a(t,x)-b(t,x)v,$
with 
\begin{align}
a(t,x)&:=\int_{\tor^3\times\r^3}\psi(|x-y|)wf(t,y,w)dydw, \label{a.p2} \\
b(t,x)&:=\int_{\tor^3\times\r^3}\psi(|x-y|)f(t,y,w)dydw. \label{b.p2}
\end{align}
System (\ref{sys}) is supplemented by initial data $u_0$ and $f_0$ for the velocity field and distribution function, respectively.

Throughout the paper we assume without a loss of generality that the total mass of $f_0$ i.e. $\int_{\tor^d\times \r^d}f_0dxdv = 1$ which due to the conservation of mass means that the total mass of the particles is always equal to $1$ and thus may dissapear in the computations.

For non-negative and integrable functions $f$ we denote:
\begin{align*}
M_\alpha f(t)&:=\int_{\tor^d\times\r^d}|v|^\alpha f(t,x,v)dxdv, & \quad 
m_\alpha f(t,x)&:=\int_{\r^d}|v|^\alpha f(t,x,v)dv,
\end{align*}
with an obvious remark that $M_0f = \|f\|_{L^1} = 1$ and that for $1\leq q\leq\infty$,
\begin{align}\label{momsupp.p2}
m_\alpha f(t,x)\leq C(R)\|f(t,x,\cdot)\|_q,
\end{align}
provided that ${\rm supp}f(t,x,\cdot)\subset B(R)$, where $B(R)$ is a ball centered at $0$ with radius $R$.
Note that
\begin{align}
\|a\|_\infty  \leq c_6M_1f,\qquad &
\|b\|_\infty  \leq c_6M_0f, \label{2.8d}
\end{align}
hence
\begin{align}\label{oszacF.p2}
|\FCS(f)(t,x,v)|\leq \|a\|_\infty + |v| \ \|b\|_\infty\leq c_6(M_1f+|v|M_0f),
\end{align}
and
$\dv \FCS(f)(t,x,v) = -db(t,x).$


\subsection{Weak formulation}\label{weak.p2}


First, let us fix the physical space dimension $d=3$. We introduce the basic  function spaces. 
\begin{align*}
L^2_{div}(\tor^3) :=&\{\omega\in L^2(\tor^3): \di\omega=0\},\\ \dot{W}^{1,p}_{div}(\tor^3) :=&\{\omega\in {\mathcal D}^{'}(\tor^3): \nablax \phi\in L^p(\tor^3),\ \di\omega=0\},\\
{W}^{1,2}_{div}(\tor^3) :=&\{\omega\in W^{1,2}(\tor^3):  \di\omega=0\},\\
{\mathcal H} :=& L^\infty(0,T;\dot{W}^{1,p}_{div}(\tor^3))\cap C([0,T];L^2_{div}(\tor^3))\cap L^2(0,T;W^{2,2}(\tor^3))\cap \\ & \cap L^\infty(0,T;W^{1,2}(\tor^3)) \cap L^p(0,T;\dot W^{1,3p}(\tor^3)),\\
{\mathcal X} :=& L^\infty((0,T)\times\tor^3\times\r^3)\cap L^\infty(0,T;L^1(\tor^3\times\r^3)).
\end{align*}
The spaces are endowed with the standard norms coming from definitions.

Next, we define  weak solutions to (\ref{sys}).

\begin{defi}\label{sol.p2}
Let $p\geq\frac{11}{5}$ and $T>0$. The couple $(f,u)$ is a weak solution of (\ref{sys}) on the time interval $[0,T)$ if and only if the following conditions are satisfied:
\begin{enumerate}
\item[(i)] $f\geq 0$, $f\in {\mathcal X}$ and $M_2 f\in L^\infty([0,T])$; 
the function $v\mapsto f(t,x,\cdot)$ is compactly supported for a.a. $t\in[0,T]$ and $x\in\tor^3$.
\item[(ii)] $u\in{\mathcal H}$ and $\partial_t u\in L^2(0,T;L^2(\tor^3))$.
\item[(iii)] For all $\phi\in C^1_{b}([0,T)\times\tor^3\times\r^3)$ such that $\phi|_{t=T}\equiv 0$, the following identity holds (the lower index ${b}$ means that the function is bounded on $\tor^3\times\r^3$)
\begin{align*}
\int_0^T\int_{\tor^3\times\r^3} f[\partial_t\phi+v\cdot\nablax\phi+F(f)\cdot\nablav \phi]dxdvdt=-\int_{\tor^3\times\r^3} f_0\phi(0,\cdot,\cdot)dxdv.
\end{align*}
\item[(iv)] For all $\varphi\in W^{1,2}(\tor^3)\cap\dot{W}^{1,p}_{div}(\tor^3)$
\begin{align*}
\int_{\tor^3}\left[\frac{\partial u}{\partial t}\cdot \varphi+ (u\cdot\nablax)u\cdot \varphi+ \tau(Du):D(\varphi)\right]dx = -\int_{\tor^3\times\r^3}(u-v)\cdot\varphi fdxdv
\end{align*}
is satisfied a.e. in [0,T] and $\lim_{t\to 0^+} u(t,\cdot) = u_0$ in $L^2(\Omega)$.
\end{enumerate}
\end{defi}

\begin{rem}\rm\label{dupa.p2}
In Definition \ref{sol.p2}, regularity of $f$ (particularly, the boundedness of $M_2f$) enable us to test in $(iv)$ with $\phi = |v|^\alpha$ for $\alpha\in[0,2]$. This observation will be useful in the large-time behavior part of the paper.
\end{rem}


%
%

\subsection{Main result}
We present the main results of the paper.

\begin{theo}\label{exiful.p2}
Let $p\geq\frac{11}{5}$ and $T>0$. Suppose that the initial data $(f_0,u_0)$ satisfy
\begin{enumerate}
\item [(i)] $0\leq f_0\in (L^1\cap L^\infty)(\tor^3\times\r^3)$, ${\rm supp}f_0(x,\cdot)\subset B(R)$ for some $R>0$ and a.a. $x\in\tor^3$, where $B(R)$ is a ball centred at $0$ with radius $R$,
\item [(ii)] $u_0\in W^{1,2}_{div}(\tor^3)$.
\end{enumerate}
Then there exists a solution of (\ref{sys}) in the sense of Definition \ref{sol.p2}.
\end{theo}

\begin{rem}[Energy inequality and conservation of momentum]\rm
Solutions to \eqref{sys} satisfy the following energy estimate:
\begin{align}\label{ee}
\big(M_2f &+ \|u\|_2^2\big)(t) + c_1\kappa\int_0^t\|\nablax  u\|_p^p ds + \int_0^t\int_{\tor^3\times\r^3} |u - v|^2f dxdvds 
\leq \big(M_2f + \|u\|_2^2\big)(0).
\end{align}
To see it on the formal level, one needs to add two instances of $\eqref{sys}_2$ tested with $u$ to $\eqref{sys}_1$ tested with $|v|^2$ and use \eqref{tau2.p2}. Estimate \eqref{ee} is a crucial part of our considerations 
and is rigorously proved in Sections \ref{33} and \ref{glob}.
Moreover, \eqref{sys} conserves the momentum:
\begin{align}\label{cm}
\frac{d}{dt}\big(\int_{\tor^3}udx + M_1f\big) = 0.
\end{align}
Indeed, integrating $\eqref{sys}_2$ over $\tor^3$, by integration by parts and thanks to $\eqref{sys}_3$ we have
\begin{align}\label{cm1}
\frac{d}{dt}\int_{\tor^3}udx = -\int_{\tor^3\times\r^3}(u-v)fdxdv.
\end{align}
On the other hand testing $\eqref{sys}_1$ with $v$ reveals that
\begin{align*}\label{cm2}
\frac{d}{dt}M_1 f = \int_{\tor^3\times\r^3}(u-v)fdxdv.
\end{align*}
Here we use the fact that
\begin{align}
\int_{\tor^3\times\r^3}\FCS(f)fdxdv = 0,
\end{align}
which is easy to see by Fubini's Theorem thanks to the anti-symmetry of $\psi(|x-y|)(w-v)$ with respect to change of variables $(x,v)$ with $(y,w)$. Adding \eqref{cm1} to \eqref{cm2} leads to \eqref{cm}.
\end{rem}

\begin{rem}\rm\label{uw.p2}
Assumption (i) in the above theorem immediately implies that
$M_\alpha f_0\leq C$ 
for some positive constant $C$ and all $\alpha\geq 0$ (from the point of view of Definition \ref{sol.p2} we need at least $M_2f_0\leq C$). In fact we could replace the boundedness of the support of $f_0$ by the assumption that
$M_5 f_0\leq C$. Then instead of working with local second apriori estimate in Section \ref{33}, it is possible to put the weight of the proof onto estimating $\frac{d}{dt}M_\alpha f$ for $\alpha\in[2,5]$ to obtain local existence. Then global existence follows from the first apriori estimate similarly to what we do in Section \ref{glob}. This approach is viable but seems more involved.
\end{rem}

The second result concerns the time-asymptotic behavior of solutions to \eqref{sys}. We express the asymptotics in the language introduced in \cite{bae2}, where the authors introduced the functional ${\mathcal E}$ that
measures the deviation of the velocity of the fluid and the velocity of the particles from their average velocities. The functional is defined as follows
\begin{align}
{\mathcal E}(t) = 2{\mathcal E}_p(t) + 2{\mathcal E}_f(t) + {\mathcal E}_d(t), \quad {\mathcal E}_p(t) = \int_{\tor^3\times\r^3}|v-v_c(t)|^2fdxdv,\label{e}\\
 {\mathcal E}_f(t) = \int_{\tor^3}|u-u_c(t)|^2dx, \quad {\mathcal E}_d(t) = |u_c(t)-v_c(t)|^2,\nonumber
\end{align}
where
\begin{align*}
u_c(t) = \int_{\tor^3}udx, \quad \mbox{and}\quad v_c(t) = \frac{\int_{\tor^3\times\r^3}vfdxdv}{\int_{\tor^3\times\r^3}fdxdv} = M_1 f.
\end{align*}

\begin{theo}\label{t-a} Suppose that $T, f_0, u_0$ satisfy
\begin{align*}
T\in(0,\infty), \quad {\mathcal E}(0)<\infty.
\end{align*}
Then the solution to \eqref{sys} in the sense of Definition \ref{sol.p2} satisfies the following exponential estimate:
\begin{align}\label{exp}
{\mathcal E}(t)\leq {\mathcal E}(0)e^{-\gamma t}, \quad t\in[0,T),
\end{align}
where $\gamma:=\min\{2\psi(\sqrt{2})+\frac{2\eta}{1+\eta}, c_1\kappa\varpi, \frac{4\eta}{1+\eta}\}$ and $\eta$ (a positive constant depending on $T$) is equal to $\eta:= \frac{c_1\kappa\varpi}{2\sup_{t\leq T}\|m_0f\|_\infty}$. 
Here $\kappa$ is the constant from Korn's inequality and $\varpi$ is the constant from Poincare's inequality for the torus $\tor^3$.

Moreover, if $p>3$ then \eqref{exp} holds for ${\mathcal E}(t) \to 0$ as $t\to \infty$.
\end{theo}

To better understand the meaning of Theorem \ref{t-a} we shall look at the ${\mathcal E}(t)$ as a Lyapunov functional. Decay (\ref{exp}) shows that $\mathcal{E}(t)$ decreases in time, although it seems that the result is local, since
the supremum norm of $m_0f$ may increase in time (but for all $T$ it is well defined). As $p$ is greater than the dimension thanks to Sobolev imbeddings we are able to show that, indeed, the energy described by $\mathcal{E}(t)$ vanishes to zero
as time goes to infinity.


%
%
%
%

\section{Existence of solutions}

Our first goal is to prove   Theorem \ref{exiful.p2}. The idea of proving existence is based on analysis of an approximative system and suitable application of Schauder theorem.

\subsection{Regularized system}\label{secregul.p2}
Note first there is a need to control the support of $f$ in $\nu$ in
the external force 
$G_{ext}(t,x)=\int_{\r^3}(u(t,x)-v)f(t,x,v)dv$. 
We introduce a cut-off function $\gamma_\epsilon:\r^3\to\r$, such that  the support in $v$ of $f$ is contained in  a ball of the radius $\frac{1}{\epsilon}$. Then we define
\begin{align*}
G_{\epsilon}(t,x)=\int_{\r^3}(\theta_\epsilon * u(t,x)-v)\gamma_\epsilon(v) f(t,x,v)dv,
\end{align*} 
where $\gamma_\epsilon\in C^\infty(\r^3)$,
${\rm supp}\,\gamma_\epsilon\subset B({1}/{\epsilon}),$ $0\leq\gamma_\epsilon\leq 1,$ $\gamma_\epsilon = 1\ {\rm on}\ B({1}/{2\epsilon}),$  $\gamma_\epsilon\to 1\ {\rm as}\ \epsilon\to 0^+,$
and $\theta_\epsilon$ is the standard mollifier i.e.
$\theta_\epsilon(x):=\frac{1}{\epsilon^3}\theta\left(\frac{x}{\epsilon}\right),$
for some $0\leq\theta\in C^\infty_0(\tor^3)$ with $\int_{\tor^3}\theta dx=1$. We further regularize also the drag force in the CS equation. For $\epsilon >0$ we denote the regularized force  $F_\epsilon(f_\epsilon)$, where 
\begin{align*}
F_\epsilon(f_\epsilon;u):= F_{CS}(f)+ (\theta_\epsilon * u - v)\gamma_\epsilon.
\end{align*}
We now write down the regularized system. For $\epsilon >0$ we consider
\begin{align}\label{regul.p2}
\left\{
\begin{array}{rcl}
\partial_t f_{\epsilon} + v\cdot \nablax f_{\epsilon}+\dv [F_\epsilon(f_\epsilon;u_\epsilon)f_\epsilon] &=& 0,\\
\partial_t u_{\epsilon} + (u_{\epsilon}\cdot\nablax)u_{\epsilon} + \nablax\pi_{\epsilon} - \di(\tau(Du_{\epsilon})) &=& -\int_{\r^d}f_{\epsilon}(\theta_\epsilon * u_{\epsilon}-v)\gamma_\epsilon dv,\\
\di u_{\epsilon} &=& 0,
\end{array}
\right.
\end{align}
with a smooth, compactly supported (in the variable $\nu$ in $B(R)$) initial data $f_{0,\epsilon}$, where $0\leq f_{0,\epsilon}\to f_0$ strongly in $L^p(\tor^3\times\r^3)$ for all $p\geq 1$ and weakly * in $L^\infty(\tor^3\times\r^3)$,  $u_{0,\epsilon}= u_0$.

To solve the regularized problem (\ref{regul.p2}) we  apply the following {\bf Schauder type scheme}.
Given $\epsilon >0$ and for suitably chosen $T$ (defined in Proposition \ref{estimun.p2}) we define a set
\begin{equation}\label{defV}
 V =\{ u \in L^\infty(0,T;L^2(\tor^3)): \di u=0 \mbox{ \ and \ } \|u\|_{L^\infty(0,T;L^2(\tor^3))} \leq \underline{C}\},
\end{equation}
where $\underline{C}$ is a chosen constant greater than $\|u_0\|_2$. 

We take a function $\bar u \in V$. 

We define 
\begin{align*}
\bar F_d=(\theta_\epsilon*
\bar u-v)\gamma_\epsilon,
\end{align*}
which is at this point a given function. Next we solve the Vlasov-type equation:
\begin{align}\label{pomocf.p2}
\partial_t f+v\cdot\nablax f+\dv[(F_{CS}(f)+ \bar F_d)f]=0,
\end{align} 
with initial datum $f(0,x,v)=f_{0,\epsilon}(x,v)$.

Next, we define $u$ as the solution of the system
\begin{align}\label{pomocu.p2}
\left\{
\begin{array}{rcl}
\partial_t u+(u\cdot\nablax)u +\nablax\pi -\di \tau(u) &=& \bar G=-\int_{\r^d}f(\theta_\epsilon *\bar u-v)\gamma_\epsilon dv,\\
\di u
 &=& 0,
\end{array}
\right.
\end{align}
with the initial datum
$u(0,x)=u_{0}(x)$
noting that in this system, the right-hand side depends on $f$ and $\bar u$, which are at this point given functions. Thus, in fact,  we solve (\ref{ns}) with a given external force. 


Existence of $f$ and $u$ is guaranteed by the following propositions belonging to the classical theory.

\begin{prop}\label{pomoc2.p2}
Let $T>0$. There exists a solution in the sense of Definition \ref{sol.p2} to the problem
\begin{align}\label{pro.p2}
\partial_t f + v\cdot\nablax f + \dv[(F_{CS}(f)+(\theta_\epsilon*u-v))\gamma_\epsilon(v)f] = 0,
\end{align}
as long as $0\leq f_0\in C^\infty(\tor^3\times\r^3)$ is compactly supported in $v$ and $u\in L^\infty(0,T;L^2_{div}(\tor^3))$. This solution $f$ belongs to the space $C^2([0,T]\times\tor^3\times\r^3)$. Moreover,
\begin{align}
\|f\|_{L^\infty(0,T;(L^\infty\cap L^1)(\tor^3\times\r^3))} \leq C,\qquad & \|f\|_{C^1} \leq C(\epsilon),\label{finfty.p2}
\end{align}
where $C$ is a positive constant depending on $\|f_0\|_{L^1(\tor^3\times\r^3)}$ and $\|f_0\|_{L^\infty(\tor^3\times\r^3)}$, while $C(\epsilon)$ depends also 
on $\epsilon$ and $\|u\|_{L^\infty(0,T;L^2(\tor^3))}$ (both constants depend also on $T$). Furthermore, $f\geq 0$ in $[0,T]\times\tor^3\times\r^3$, provided $f_0\geq 0$ in $\tor^3\times\r^3$. 
\end{prop}

\begin{proof}
This proposition along with its proof can be found in \cite[Appendix A]{Bae}. It is based on the fact that both $F_{CS}(f)$ and $F_d$ in (\ref{pro.p2}) are smooth. 
Local existence in Proposition \ref{pomoc2.p2} is showed by a standard method of characteristics combined with a fixed point argument. Then to conclude the global existence, 
a priori $C^1$ estimate for $f$ is derived. It can be done because the nonlinearity in \eqref{pro.p2} that comes from the multiplication by $F_{CS}(f) + F_d$ is smooth (here regularity of 
the communication weight $\psi$ and the mollifier $\theta_\epsilon$ play the crucial role). 
\end{proof}

\begin{prop}\label{pomoc1.p2}
Let $p\geq \frac{11}{5}$ and $T>0$. There exists a unique solution in the sense of Definition \ref{sol.p2} to the problem
\begin{align*}
\partial_t u + (u\cdot\nablax)u + \nablax\pi -\di (\tau(Du)) = G,
\end{align*}
provided $u_0\in W^{1,2}_{0,div}(\tor^3)$ and $G \in L^2(0,T;L^2(\tor^3))$. Moreover,
\begin{align*}
\|u\|_{\mathcal H} \leq C,& \qquad 
\|\partial_t u\|_{L^2(0,T;L^2(\tor^3))} \leq C,
\end{align*}
where $C$ is a positive constant depending on $\|u_0\|_{W^{1,2}(\tor^3)}$, $\|G\|_{L^2(0,T;L^2(\tor^3))}$, $p$ and $T$.
\end{prop}

\begin{proof}
The proof can be found in \cite[Theorem 4.5]{mal}. The proof is based on the structure of $\tau(\cdot)$. We may consider different types of approximations for which it is not difficult to construct solutions. To obtain a-priori estimates allowing to pass from the approximate problem to the original one, we first test the approximate problem by the velocity. The pressure and the convective term disappear due to the divergence-free condition and the time derivative and the stress tensor (property \eqref{tau2.p2}) yield the estimates of the velocity in $L^\infty(0,T;L^2(\tor^3))$ and in $L^p(0,T;W^{1,p}(\tor^3))$. Next step consists in testing by $(1+\|\nablax u(t)\|^{-\lambda}_{L^2(\tor^3)})\Delta u$ for suitable $\lambda \in [0,1]$. The time derivative and the structure of the stress tensor (more precisely, property \eqref{tau4.p2}) provide now estimates in $L^\infty(0,T;W^{1,2}(\tor^3))$, $L^p(0,T;W^{1,3p}(\tor^3))$  and $L^2(0,T;W^{2,2}(\tor^3))$. However, the convective term does not disappear now and we need to control a term of the form $\sim |\nablax u|^3$ on the right hand-side, using estimates from the first step together with the form of the left-hand side. It is possible to estimate the cubic term for $p\geq \frac{11}{5}$ and get
the following bound 
\begin{equation}\label{hen}
 \sup_{t<T} \|u\|_{W^{1,2}(\tor^3)}^2 + \|\nablax^{(2)} u\|_{L^2(0,T;L^2(\tor^3))}^2 + \|\nablax u\|_{L^{p}(0,T; L^{3p}(\tor^3))}^p \leq C(\|G\|_{L_2(\tor^3 \times (0,T))}^2 +\|u_0\|_{W^{1,2}(\tor^3)}^2),
\end{equation}
where the constant $C$ depends also on the estimates from the first step, i.e. on the norms of the velocity in $L^\infty(0,T;L^2(\tor^3))$ and in $L^p(0,T;W^{1,p}(\tor^3))$. Moreover, the velocity $u$ can be used as a test function in the weak formulation which allows to prove the uniqueness of the solution. Finally, using as a test function $\partial_t u$, we deduce the estimates of $u$ in $L^\infty(0,T;\dot{W}^{1,p}(\tor^3))$ and $\partial_t u$ in $L^2((0,T)\times \tor^3)$.

In a sense, we repeat the idea of the proof for the two dimensional Navier--Stokes system, but with better integrability given by the features of $\tau(\cdot)$ for $p\geq 11/5$.
\end{proof}

The coupling is realized by force $\bar G$.
Let \mbox{$u\in L^\infty(0,T;L^2_{div}(\tor^3))$} and $f\in  C^1([0,T]\times\tor^3\times\r^3)$, then
\begin{align}
\int_0^T\|\bar G\|_{L^2(\tor^3)}^2&=\int_0^T\int_{\tor^3}\left|\int_{\r^3}(\theta_\epsilon * \bar u-v)\gamma_\epsilon f dv\right|^2dxdt\nonumber\\
&=\int_0^T\int_{\tor^3}\left|\int_{B(\frac{1}{\epsilon})}(\theta_\epsilon * \bar u-v)\gamma_\epsilon
f dv\right|^2dxdt\nonumber\\
&\leq C(T,\epsilon)\int_0^T\int_{\tor^3}\int_{B(\frac{1}{\epsilon})}|\theta_\epsilon * \bar u -v|^2|f|^2dvdxdt\nonumber\\
&\leq C(T,\epsilon)\|f\|_{L^\infty((0,T)\times \tor^3\times \r^3)}^2(\|\bar u\|_{L^2((0,T)\times \tor^3)}^2+1).\label{ext.p2}
\end{align}
Therefore $\bar G$ belongs to $L^2(0,T;L^2(\tor^3))$ with its norm depending on $T,\epsilon$, $\|f\|_{L^\infty((0,T)\times \tor^3\times \r^3)}$ and $\|u\|_{L^\infty(0,T;L^2(\tor^3))}$. 
Thus, by Proposition \ref{pomoc1.p2}, there exists a unique $u$ --- a solution to (\ref{pomocu.p2}) in the sense of Definition \ref{sol.p2}.
Existence of a unique $f$ --- a solution to (\ref{pomocf.p2}) belonging additionally to the space $C^1([0,T]\times\tor^3\times\r^3)$ --- follows then by Proposition \ref{pomoc2.p2}. 
Note that $C(T,\epsilon)$ decreases to zero as $T\to 0$.

Definition of $f$ and $u$ constructed by $\bar u$ determines a map $\mathcal{T}:V \to L^\infty(0,T;L^2(\tor^3))$ such that $\mathcal{T}(\bar u) =u$, where $u$ is the solution to (\ref{pomocu.p2})
and $f$ to (\ref{pomocf.p2}). We need to show $\mathcal{T}$ maps $V$ into itself, it is continuous and compact. There is no need to explain the set $V$ is convex.

\subsection{Compactness}

Our next step is to prove that map $\mathcal{T}$ has a fixed point in $V$ and it defines a solution of \eqref{regul.p2}.
We begin with estimates for $u$ and $f$ in ${\mathcal H}$ and ${\mathcal X}$, respectively.

\begin{prop}\label{estimun.p2}
Given $\epsilon >0$, let $\underline{C} > \|u_0\|_2$ and $\bar u\in V$. Then there exists $T>0$ such that
\begin{equation}
 \|u\|_{L^\infty(0,T;L^2(\tor^3))}\leq \underline{C}.
\end{equation}
Moreover, there exist positive constants $C(\epsilon)$ and $C$ such that $\{u,f\}$, satisfy the following bounds:
\begin{align*}
&(i)&\|u\|_{\mathcal H}\leq C(\epsilon), \\
&(ii)&\|\partial_t u\|_{L^2(0,T;L^2(\tor^3))}\leq C(\epsilon),\\
&(iii)&\|f\|_{\mathcal X}\leq C,\\
&(iv)& \|f\|_{C^1}\leq C(\epsilon),
\end{align*}
where $C$ is independent of  $\epsilon$. Moreover, there exists a non-decreasing 
function ${\mathcal R}_\epsilon:[0,T]\to[0,\infty)$ such that
\begin{align}\label{supest.p2}
{\rm supp}f(t,x,\cdot)\subset B({\mathcal R}_\epsilon(t)),\ \ \ {\rm for\ all}\ t\ {\rm and\ a.a}\ x.
\end{align}
\end{prop}

\begin{proof}
By Proposition \ref{pomoc1.p2} and the definition of $u$ it is clear that to obtain estimate of $u$ in ${\mathcal H}$  it suffices to estimate $\|\bar G\|_{L^2(0,T;L^2(\tor^3))}$.
By testing the weak formulation  by $u$ (which by Proposition \ref{pomoc1.p2} is a suitable test function), applying Korn's inequality  and (\ref{tau2.p2}) we obtain
\begin{align*}
\frac{1}{2}\frac{d}{dt}\|u\|_2^2+ c_1\kappa\|\nablax u\|_p^p\leq \|u\|_2^2 + \|\bar G\|_2^2,
\end{align*}
which by inequality (\ref{ext.p2}) and (\ref{finfty.p2}) implies that
\begin{align*}
\frac{1}{2}\frac{d}{dt}\|u\|_2^2+ c_1\kappa\|\nablax u\|_p^p \leq \|u\|_2^2+C(T,\epsilon)\|\bar u\|_2^2 + C(T,\epsilon).
\end{align*}
Therefore by Gronwall's lemma there exists $T$ depending on $\epsilon$, such that
\begin{align*}
\|u\|_{L^\infty(0,T;L^2(\tor^3))} \leq \underline{C},
\end{align*}
Using Proposition \ref{pomoc1.p2} we finish the proof of $(i)$. 

The proof of $(ii)$ follows similarly to the proof of $(i)$ by testing the weak formulation for $u$ with $\partial_tu$ and using the previously proved estimates.


We continue with  estimates of $f$.
The key point concerns the propagation of the support.
The estimate of the support of $f$ is proved in Lemmas \ref{nos.p2} and \ref{lemr.p2} below. Lemma \ref{nos.p2} shows that ${\mathcal R}_\epsilon(t)$ 
depends on $\|u\|_{L^2(0,T;W^{2,2}(\tor^3))}$ and $\|M_1f\|_{L^\infty(0,T)}$. On the other hand, in Lemma \ref{lemr.p2} we prove that $\|M_1f\|_{L^\infty(0,T)}$ is uniformly bounded in terms of
$\|u\|_{L^2(0,T;W^{2,2}(\tor^3))}$. Therefore, by $(i)$ from Proposition \ref{estimun.p2}, the function ${\mathcal R}_\epsilon$ is independent of $n$ but depends on $\epsilon$. 
This observation concludes the proof of (\ref{supest.p2}).

\begin{lem}[Propagation of the support of velocity]\label{nos.p2}
Let $f$ be a solution to (\ref{pro.p2}) subjected to the initial data with the support in $v$ contained in the ball $B(R)$. Then there exists a non-decreasing function ${\mathcal R}:[0,T]\to[0,\infty)$ 
such that for all $t\in[0,T]$ and almost all $x\in\tor^3$, the support of $f(t,x,\cdot):\r^3\to\r$ is contained in a ball of radius ${\mathcal R}(t)$. Moreover, for each $t\in[0,T]$ the value ${\mathcal R}(t)$ 
depends only on $t$, $\|u\|_{L^2(0,t;W^{2,2}(\tor^3))}$, $\|M_1f\|_{L^\infty(0,T)}$ and $R$. 
\end{lem}
\begin{proof}
Let $f$ be a solution to (\ref{pro.p2}). Consider the solution of the system of ODE's:
\begin{align}\label{ode.p2}
\left\{
\begin{array}{rll}
\frac{dx}{dt}(t) &= v(t), & x(0)=x_0,\\
\frac{dv}{dt}(t) &= F_{CS}(f)(t,x(t),v(t))+[(\theta_\epsilon *u)(t,x(t))-v(t)]\gamma_\epsilon(v(t)),& v(0)=v_0.
\end{array}
\right.
\end{align}
Then the function $\tilde{f}(t,x_0,v_0):=f(t,x(t),v(t))$ satisfies the equation
\begin{align*}
\partial_t \tilde{f} = -\left(\dv F_{CS}(f)+\dv [(\theta_\epsilon * u(t,x(t))-v) \gamma_\epsilon(v)]\right)\tilde f.
\end{align*}
Note that we are required to look at the three terms coming from the divergence: $\gamma_\epsilon(v)$, $\theta_\epsilon * u \nabla_v \gamma_\epsilon(v)$ and 
$v\cdot \nabla_v \gamma_\epsilon(v)$. By definition of $\gamma_\epsilon(\cdot)$ from the beginning of Section 3.1 we see that
\begin{equation*}
 |\gamma_\epsilon(v)| + |v\cdot \nabla_v \gamma_\epsilon(v)|\leq C
\end{equation*}
with $C$ independent of $\epsilon$. On other hand, using the explicit form of $\theta_\epsilon$, it is possible to compute  that $|(\theta_\epsilon * u)(t,x) | \leq \|(\theta_\epsilon*u)(t,\cdot)\|_{L_\infty}\leq C\epsilon^{-\frac{1}{2}}\|u(t,\cdot)\|_{L^2(\tor^3)}$, 
so we conclude
\begin{equation*}
 |(\theta_\epsilon * u)(t,x) \nabla_v \gamma_\epsilon(v)|\leq C \epsilon^\frac{1}{2}.
\end{equation*}
Hence, recalling $b$ is defined in (\ref{b.p2}), we obtain
\begin{equation}\label{fff}
\tilde{f}(t,x(t),v(t)) = e^{3\int_0^t(b+B)ds}f_{0}(x(t),v(t)),
\mbox{ \ \ \  where  \ }
 \|B(t)\|_{L_\infty}\leq C(\|u(t,\cdot)\|_{L^2(\tor^3)}+C)
\end{equation}
for sufficiently small $\epsilon$. Let us note at this point that the above estimate provides a proof of $L^\infty$ bound from \eqref{finfty.p2}. 

Therefore, $\tilde{f}(t,x_0,v_0)=0$ whenever $f_{0}(x_0,v_0)=0$ which implies that $f(t,x,v) = 0$ whenever the characteristic that contains point $(x,v)$ starts at $(x_0,v_0)$
such that $f_{0}(x_0,v_0)=0$. We solve $(\ref{ode.p2})_2$, to get
\begin{align*}
v(t) = e^{-\int_0^t(b(s,x(s))+B(s,x(s),v(s)))ds}\times \\
\times \left(v_0 + \int_0^te^{\int_s^t(b(r,x(r))+ B(s,x(s),v(s)) )dr}[a(s,x(s))+(\theta_\epsilon*u)(s,x(s))]\gamma_\epsilon (v(s))ds\right),
\end{align*} 
which, since by (\ref{2.8d}) $1\leq b+1\leq c_6M_0f+1$, by (\ref{finfty.p2}) $M_0f=1$, and $B$ is bounded in terms of the norm  $L^2(0,T;W^{2,2}(\tor^3))$, implies that
\begin{align*}
|v(t)| &\leq Ce^{Ct}\left(|v_0| + \int_0^t|a(x(s),s)|ds + \int_0^t\|u(s)\|_\infty ds\right)\\ 
& {\leq} Ce^{Ct}\left(|v_0|+ t\|M_1f\|_{L^\infty(0,T)} + \|u\|_{L^2(0,t;W^{2,2}(\tor^3))}\right)\\
&\leq Ce^{Ct}\left(R + t\|M_1f\|_{L^\infty(0,T)} + \|u\|_{L^2(0,T;W^{2,2}(\tor^3))}\right) =:{\mathcal R}(t),
\end{align*}
where we also used the embedding $L^2(0,t;W^{2,2}(\tor^3))\hookrightarrow L^1(0,t;L^\infty(\tor^3))$. We will underline at this stage that the estimate depends on $\epsilon$, but at the end of the proof of existence, this estimate will imply that the support of $f$ is bounded independently of $\epsilon$.
\end{proof}

\begin{lem}\label{lemr.p2}
Let $f$ be a solution to (\ref{pro.p2}) subjected to the initial data with the support in $v$ contained in the ball $B(R)$. Then
\begin{align*}
M_1f\leq C(\epsilon),
\end{align*}
for some positive $\epsilon$-dependent constant $C(\epsilon)$.
\end{lem}
\begin{proof}
First we integrate (\ref{pro.p2}) to see that $M_0f=1$. Next
we multiply (\ref{pro.p2}) by $|v|$ and integrate to get
\begin{align*}
0=\frac{d}{dt}M_1 f + \underbrace{\int_{\tor^3\times\r^3}|v|v\cdot \nablax fdxdv}_{=0} + \int_{\tor^3\times\r^3}|v|\dv\big[(F_{CS}(f)+(\theta_\epsilon*u-v)\gamma_\epsilon) f\big] dxdv\\
= \frac{d}{dt}M_1 f - \int_{\tor^3\times\r^3}\frac{v}{|v|} \cdot \big(F_{CS}(f)+(\theta_\epsilon*u-v)\gamma_\epsilon\big)fdxdv.
\end{align*}
Thus
\begin{align*}
\frac{d}{dt}M_1 f &= \int_{\tor^3\times\r^3}\frac{v}{|v|} \cdot F_{CS}(f)fdxdv + \int_{\tor^3\times\r^3}\frac{v}{|v|}\cdot \theta_\epsilon*u \gamma_\epsilon fdxdv - \int_{\tor^3\times\r^3}|v|f\gamma_\epsilon dxdv\\
&\leq \int_{\tor^3\times\r^3} |F_{CS}(f)|fdxdv + \int_{\tor^3\times\r^3}|\theta_\epsilon*u|fdxdv\\
&\stackrel{(\ref{oszacF.p2})}{\leq} CM_1fM_0f + C\|u\|_{W^{2,2}(\tor^3)}M_0f.
\end{align*}
Since $M_0f=1$, it implies that $M_1f$ is bounded on $[0,T]$ if $M_1f_0$ is finite.
\end{proof}


{\bf Proof of $(iii)$ and $(iv)$.} Even though it follows directly from Proposition \ref{pomoc2.p2}, it is worthwhile to note that the conservation of mass is trivial, while the $L^\infty$ bound was shown by \eqref{fff}.
\end{proof}

With the estimates provided by Proposition \ref{estimun.p2}  we are ready to prove the following proposition that states the existence of solutions to the regularized system (\ref{regul.p2}) and finishes the part 4 of the proof of 
Theorem \ref{exiful.p2}.

\begin{prop}\label{lim.p2}
Given $\underline{C}$ defined as in (\ref{defV}) and $T$ given by Proposition \ref{estimun.p2}. Then there exists a fixed point of $\mathcal{T}$ belonging to $V$. In addition $u$ fulfills of all regularity 
described by Proposition \ref{estimun.p2}.
\end{prop}

\begin{proof}
Continuity of map $\mathcal{T}$ follows from stability of solvability given by Propositions \ref{pomocf.p2} and \ref{pomocu.p2}. As $u \in V$ then $G \in L^2(0,T;L^2(\tor^3))$, hence $u\in \mathcal{H}$ and $u_t \in L^2(0,T;L^2(\tor^3))$. These facts imply
compactness. By Schauder's theorem we found a fixed point of map $\mathcal{T}$ belonging to $V$. The extra regularity comes from Proposition \ref{estimun.p2}. Properties of $f$ are concluded from Proposition \ref{pomocf.p2}.
We are done.

\end{proof}

\subsection{Local convergence with the regularized solutions}\label{33}

Until now we  proved existence of solutions to the regularized system (\ref{regul.p2}). The next goal is to converge with $\epsilon$ to $0$ and to obtain local-in-time existence for (\ref{sys}).

\begin{prop}\label{momprop.p2}
Let $p\geq\frac{11}{5}$ and $(f_{\epsilon},u_{\epsilon})$ be a solution to system (\ref{regul.p2}) constructed as a limit of the approximate solutions as proved in Proposition \ref{lim.p2}. Then there exists 
$T^*\in(0,T]$, such that $(f_{\epsilon},u_{\epsilon})$ satisfies the following estimates
\begin{align}
&\|M_\alpha f_{\epsilon}\|_{L^\infty[0,T^*]} \leq C(T), & {\rm for}\ \ \ 0\leq\alpha\leq 2,\label{malfa.p2}\\
&\|u_\epsilon\|_{L^\infty(0,T^*;L^2_{div}(\tor^3))\cap L^p(0,T^*;\dot{W}^{1,p}_{div}(\tor^3))} \leq C(T),\label{ualfa.p2}\\
&\left\|\int_{\r^3}(\theta_\epsilon * u_{\epsilon}-v) \gamma_\epsilon f_{\epsilon}dv\right\|_{L^2(0,T^*;L^2(\tor^3))} \leq C(T),\label{exalfa.p2} 
\end{align}
where $C(T)$ is a positive constant depending on the initial data and $T$.
\end{prop}

\begin{proof}
 We multiply equation $(\ref{regul.p2})_1$ by $|v|^2$ and integrate to obtain
\begin{align*}
0 &= \frac{d}{dt}M_2f_\epsilon + \int_{\tor^3\times\r^3}|v|^2v\cdot\nablax f_\epsilon dxdv + \int_{\tor^3\times\r^3}|v|^2\dv[F_{CS}(f_\epsilon)f_\epsilon + \gamma_\epsilon(\theta_\epsilon * u_\epsilon - v)f_\epsilon]dxdv\\
&= \frac{d}{dt}M_2f_\epsilon - 2\int_{\tor^3\times\r^3}v\cdot F_{CS}(f_\epsilon)f_\epsilon dxdv -2\int_{\tor^3\times\r^3}\gamma_\epsilon v\cdot (\theta_\epsilon * u_\epsilon - v)f_\epsilon dxdv
\end{align*}
and since by substituting $x$ with $y$ and $v$ with $w$ (as in the estimate of ${\mathcal S}_1$ in the proof of Proposition \ref{momprop.p2}) we have
\begin{align*}
\int_{\tor^3\times\r^3}v\cdot F_{CS}(f_\epsilon)f_\epsilon dxdv \leq  -C \int_{\tor^3\times\r^3} |w-v|^2\psi(|x-y|)f_\epsilon(t,w,y) f_\epsilon(t,u,x)  dydwdxdv\leq 0,
\end{align*}
we deduce the inequality
\begin{align}\label{to.p2}
\frac{d}{dt}M_2 f_\epsilon \leq 2\int_{\tor^3\times\r^3}\gamma_\epsilon v\cdot (\theta_\epsilon * u_\epsilon - v)f_\epsilon dxdv. 
\end{align}
Next we test the weak formulation for $u_\epsilon$ by $u_\epsilon$ to get
\begin{align}\label{tego.p2}
\frac{1}{2}&\frac{d}{dt}\|u_\epsilon\|_2^2 + c_1\kappa\|\nablax  u_\epsilon\|_p^p\leq - \int_{\tor^3\times\r^3} f_\epsilon (\theta_\epsilon u_\epsilon - v)\cdot  u_\epsilon \gamma_\epsilon  dxdv\\
& = \int_{\tor^3\times\r^3} f_\epsilon (\theta_\epsilon * u_\epsilon - v)\cdot  (\theta_\epsilon * u_\epsilon -u_\epsilon) \gamma_\epsilon dxdv  
- \int_{\tor^3\times\r^3} f_\epsilon (\theta_\epsilon * u_\epsilon - v)\cdot  \theta_\epsilon * u_\epsilon \gamma_\epsilon dx dv.\nonumber 
\end{align}
We add (\ref{to.p2}) and two instances of (\ref{tego.p2}) obtaining
\begin{align}\label{1}
\frac{d}{dt}\big(M_2f_\epsilon &+ \|u_\epsilon\|_2^2\big) + 2c_1\kappa\|\nablax u_\epsilon\|_p^p + 2\int_{\tor^3\times\r^3}\gamma_\epsilon |\theta_\epsilon * u_\epsilon - v|^2f_\epsilon dxdv\nonumber \\
&\leq 2 \int_{\tor^3\times\r^3}\gamma_\epsilon (\theta_\epsilon * u_\epsilon - v)\cdot (\theta_\epsilon * u_\epsilon -u_\epsilon) f_\epsilon dxdv. 
\end{align}
H\"older's and Young's inequalities  yield the following estimate of the right-hand side
\begin{align}
 2 \int_{\tor^3\times\r^3} &\gamma_\epsilon (\theta_\epsilon * u_\epsilon - v)\cdot (\theta_\epsilon * u_\epsilon -u_\epsilon) f_\epsilon dxdv\nonumber \\
& \leq \int_{\tor^3\times\r^3}\gamma_\epsilon |\theta_\epsilon * u_\epsilon - v|^2f_\epsilon dxdv + C \|\theta_\epsilon*u_\epsilon - u_\epsilon\|_6^2 \|m_0 \sqrt{f_\epsilon}\|^2_{3}\label{2}\\
& \leq \int_{\tor^3\times\r^3}\gamma_\epsilon |\theta_\epsilon * u_\epsilon - v|^2f_\epsilon dxdv + \eta \|\theta_\epsilon*u_\epsilon - u_\epsilon\|_6^p + C(\eta)\|m_0 f_\epsilon\|_{\frac 32}^\frac{p}{p-2}.\nonumber
\end{align}
By Young's inequality for convolutions we have
$\|\theta_\epsilon*u_\epsilon - u_\epsilon\|_6^p\leq 2^p\|\nablax u_\epsilon\|_p^p,$
thus choosing a suitable $\eta$ we obtain
\begin{align*}
\frac{d}{dt}\big(M_2f_\epsilon &+ \|u_\epsilon\|_2^2\big) + c_1\kappa\|\nablax u_\epsilon\|_p^p + \int_{\tor^3\times\r^3}\gamma_\epsilon |\theta_\epsilon * u_\epsilon - v|^2f_\epsilon dxdv\leq C\|m_0 f_\epsilon\|_{\frac 32}^\frac{p}{p-2}.
\end{align*}
We apply Lemma \ref{mom.p2} and  uniform $L^1$ and $L^\infty$ bounds on $f_\epsilon$ given by (\ref{finfty.p2}) to get
\begin{align*}
\|m_0 f_\epsilon\|_\frac{3}{2}\leq C(M_\frac{3}{2}f_\epsilon)^\frac{2}{3} \leq C(M_2f_\epsilon)^\frac{1}{2}
\end{align*}
which yields
\begin{align*}
\frac{d}{dt}\Big(M_2f_\epsilon &+ \|u_\epsilon\|_2^2\Big) + c_1\kappa\|\nablax u_\epsilon\|_p^p + \int_{\tor^3\times\r^3}\gamma_\epsilon |\theta_\epsilon * u_\epsilon - v|^2f_\epsilon dxdv\leq C(M_2f_\epsilon)^\frac{p}{2p-4}.
\end{align*}
By the nonlinear version of Gronwall's lemma (Lemma \ref{gron.p2}) there exists $T^*\in(0,T)$ (with $T^*=T$ for $p\geq 4$) such that for $t\in[0,T^*]$, we have
\begin{align}\label{hmm}
\big(M_2f_\epsilon &+ \|u_\epsilon\|_2^2\big)(t) + c_1\kappa\int_0^t\|\nablax  u_\epsilon\|_p^p ds + \int_0^t\int_{\tor^3\times\r^3} |\theta_\epsilon * u_\epsilon - v|^2\gamma_\epsilon f_\epsilon dxdv ds \leq C(T).
\end{align}
This proves \eqref{malfa.p2} and \eqref{ualfa.p2}. It remains to prove \eqref{exalfa.p2}. We do this by using \eqref{hen}: we estimate
\begin{align*}
G_\epsilon = \int_{\r^3}(\theta_\epsilon*u_\epsilon-v)\gamma_\epsilon f_\epsilon dv
\end{align*}
in $L^2(0,T^*;L^2(\tor^3))$ by a combination of terms that are bounded thanks to the energy estimate \eqref{hmm} and terms that appear on the left-hand side of \eqref{hen}; then we move the bad terms to the 
left hand side of \eqref{hmm} and finish the estimation. We perform the computations only for $p<3$. For $p\geq 3$ slightly different argument is needed, but due to higher regularity of solutions, the proof is  in fact easier. 
We have
\begin{align*}
\int_0^{T^*}\|G_\epsilon\|_2^2dt\leq C\left(\int_0^{T^*}\int_{\tor^3}\left(\int_{\r^3}|\theta_\epsilon * u_\epsilon|\gamma_\epsilon f_\epsilon dv\right)^2dxdt + 
\int_0^{T^*}\int_{\tor^3}\left(\int_{\r^3}|v|\gamma_\epsilon f_\epsilon dv\right)^2dxdt\right) =: A + B.
\end{align*}
First we estimate $A$:

\begin{align}\label{push1}
A \leq \int_0^{T^*}\int_{\tor^3}|\theta_\epsilon*u_\epsilon|^2 (m_0f_\epsilon)^2dxdt
\leq \int_0^{T^*}\|u_\epsilon\|_\infty^2dt\cdot\sup_{t\leq {T^*}}\int_{\tor^3}(m_0f_\epsilon)^2dx.
\end{align}
Lemma \ref{mom.p2} implies that
\begin{align}\label{push2}
\int_{\tor^3}(m_0f_\epsilon)^2dx \leq \int_{\tor^3}(m_0f_\epsilon)^\frac{5}{3}dx\cdot \|m_0f_\epsilon\|_\infty^\frac{1}{3}\leq CM_2f_\epsilon \|f_\epsilon\|_\infty^\frac{1}{3}{\mathcal R}\stackrel{\eqref{finfty.p2},\eqref{hmm}}
{\leq}C(T^*){\mathcal R},
\end{align}
where ${\mathcal R}$ denotes the radius of the support of $f_\epsilon$ in $v$.
Moreover, the proof of Lemma \ref{nos.p2} implies that
\begin{align}\label{push3}
{\mathcal R}\leq Ce^{C{T^*}}\left(R + {T^*}\sqrt{M_2f_\epsilon} + \int_0^{T^*}\|u_\epsilon\|_\infty dt\right),
\end{align}
thus combining \eqref{push1} and \eqref{push2} with the energy estimate \eqref{hmm} we get
\begin{align}\label{push4}
A \leq C(T) \int_0^{T^*}\|u_\epsilon\|_\infty^2dt\left(\sqrt{M_2f_\epsilon} + \int_0^{T^*}\|u_\epsilon\|_\infty dt + 1\right)
\leq C(T)\left[\left(\int_0^{T^*}\|u_\epsilon\|_\infty^2dt\right)^\frac{3}{2}+1\right].
\end{align}
To estimate $\|u_\epsilon\|_\infty$, we use Gagliardo--Nirenberg inequality obtaining
\begin{align*}
\|u_\epsilon\|_\infty\leq C\Big(\|\nablax u_\epsilon\|_{3p}^\frac{3-p}{2} + \|u_\epsilon\|_2^{\frac{3-p}{p}}\Big) \|u_\epsilon\|_\frac{3p}{3-p}^\frac{p-1}{2} \\
\leq C\Big(\|\nablax u_\epsilon\|_{3p}^\frac{3-p}{2} + \|u_\epsilon\|_2^{\frac{3-p}{p}}\Big) \Big(\|\nablax u_\epsilon\|_p^\frac{p-1}{2} +\| u_\epsilon\|_2^\frac{p-1}{2}\Big),
\end{align*}
which implies that (note that $\frac{p}{3-p}>1$ for $p>\frac{3}{2}$)
\begin{align*}
\int_0^{T^*}\Big(\|\nablax u_\epsilon\|_{3p}^\frac{3-p}{2} + \|u_\epsilon\|_2^{\frac{3-p}{p}}\Big) \Big(\|\nablax u_\epsilon\|_p^\frac{p-1}{2} +\| u_\epsilon\|_2^\frac{p-1}{2}\Big) dt\\
\stackrel{H(\frac{p}{3-p})}{\leq}C\left(\int_0^{T^*}\|\nablax u_\epsilon\|_{3p}^pdt\right)^\frac{3-p}{p}\left(\int_0^{T^*}\|\nablax u_\epsilon\|_p^{p\cdot\frac{p-1}{2p-3}}dt\right)^\frac{2p-3}{p} + l.o.t.,
\end{align*}
where l.o.t. denotes lower order terms connected with the presence of $\|u_\epsilon\|_2$ which is bounded on $(0,{T^*})$. 
Therefore
\begin{align*}
A\leq C(T)\left(\int_0^{T^*}\|\nablax u_\epsilon\|_{3p}^pdt\right)^\frac{9-3p}{2p}\left(\int_0^{T^*}\|\nablax u_\epsilon\|_p^{p\cdot\frac{p-1}{2p-3}}dt\right)^\frac{6p-9}{2p} + l.o.t.
\end{align*}

We use the fact that $\frac{p-1}{2p-3}\leq 1$ for $p\geq 2$ and Young's ineqaulity with exponent $\frac{2p}{9-3p}$ (which is greater than 1 for $p>\frac{9}{5}$) to get for arbitrary $\eta>0$

\begin{align}\label{push5}
A\leq \eta\int_0^{T^*}\|\nablax u_\epsilon\|_{3p}^pdt + C(\eta)\left(\int_0^{T^*}\|\nablax u_\epsilon\|_p^pdt\right)^\frac{6p-9}{5p-9} + C.
\end{align}

On the other hand for $B$, we have by Lemma \ref{mom.p2} 
\begin{align*}
B\leq \int_0^{T^*}\int_{\tor^3}\left(m_1f_\epsilon\right)^\frac{5}{4}\cdot \left( m_1f_\epsilon \right)^\frac{3}{4}\leq TM_2f_\epsilon\|m_1f_\epsilon\|_\infty^\frac{3}{4}
\stackrel{\eqref{finfty.p2},\eqref{hmm}}{\leq} C(T){\mathcal R}^3. 
\end{align*}
We estimate ${\mathcal R}$ again using \eqref{push3} obtaining
\begin{align*}
B\leq C(T)\left[\left(\int_0^{T^*}\|u_\epsilon\|_\infty dt\right)^3+1\right]\leq C(T)\left[\left(\int_0^{T^*}\|u_\epsilon\|_\infty^2 dt\right)^\frac{3}{2}+1\right]
\end{align*}
and this is the estimate with exactly the same right-hand side as \eqref{push4}. Thus from this point we proceed like in the estimation of $A$ altogether getting
\begin{align}\label{push6}
\int_0^{T^*}\|G_\epsilon\|_2^2dt\leq A+B\leq \eta\int_0^{T^*}\|\nablax u_\epsilon\|_{3p}^pdt + C(\eta)\left[\left(\int_0^{T^*}\|\nablax u_\epsilon\|_p^pdt\right)^\frac{6p-9}{5p-9}  + 1\right]
\end{align}
for all $\eta>0$ and note that due to energy estimate \eqref{hmm} the second summand on the right-hand side is bounded on $[0,T^*]$. We apply this estimate to \eqref{hen} which after taking a sufficiently small $\eta$ 
enables us to move the term with $\|\nablax u_\epsilon\|_{3p}^p$ to the left-hand side which results in
\begin{align}\label{push7}
\sup_{t\leq {T^*}}\|u_\epsilon\|^2_{W^{1,2}(\tor^3)}+ \|\nablax^{(2)}u_\epsilon\|_{L^2(0,{T^*};L^2(\tor^3))}^2 +\frac{1}{2}\|\nablax u_\epsilon\|_{L^p(0,{T^*};L^{3p}(\tor^3))}^p\leq C,
\end{align}
where $C$ depends on $\|u_{0,\epsilon}\|_2^2, M_2f_{0,\epsilon}, R, \|f_\epsilon\|_\infty, T^*, T, p$, all of which are either fixed or bounded independently of $\epsilon$. Finally applying \eqref{push7} to \eqref{push6} 
finishes the proof of \eqref{exalfa.p2}.
\end{proof}

The following corollary combines all the necessary local in time uniform estimates of $u_{\epsilon}$ and $f_{\epsilon}$ proved throughout this section.

\begin{coro}\label{bounds.p2}
The solutions $(f_{\epsilon},u_{\epsilon})$ satisfy estimates $(i)-(iii)$ from Proposition \ref{estimun.p2} and all estimates from Proposition \ref{momprop.p2} uniformly with respect to $\epsilon>0$ on the time interval $[0,T^*]$. 
\end{coro}

\begin{proof}
To prove estimate $(i)$ we notice that by Proposition \ref{pomocu.p2} the $\|\cdot\|_{{\mathcal H}}$ norm of $u^\epsilon_n$ depends only on $\|u_0\|_{W^{1,2}(\tor^3)}$, which is fixed and on $\|G\|_{L^2(0,T;L^2(\tor^3))}$ which by 
Proposition \ref{momprop.p2} is uniformly bounded with respect to  $\epsilon$ on the time interval $[0,T^*]$. Therefore also $\|u^\epsilon\|_{{\mathcal H}}$ is uniformly bounded on $[0,T^*]$. The exactly same argument is valid 
for the $\epsilon$-independent estimate $(ii)$. Estimate $(iii)$ was already proved to be $\epsilon$-independent.
It remains to show that $f_\epsilon$ satisfies $(\ref{supest.p2})$ with  ${\mathcal R}$ independent of $\epsilon$. By Lemmas \ref{nos.p2} and \ref{lemr.p2} each iterative solution $f_\epsilon$ has a support contained in a ball of 
radius ${\mathcal R}_\epsilon$ with ${\mathcal R}_\epsilon(t)$ depending on $\|u_\epsilon\|_{L^2(0,T;W^{2,2}(\tor^3))}$ and $\|M_1f_\epsilon\|_\infty$ (and $R$ which depends only on the initial data). However, by Proposition \ref{momprop.p2},
these quantities are uniformly bounded on $[0,T^*]$ thus so is ${\mathcal R}_\epsilon$. 
 \end{proof}

\begin{proof}[Proof of Theorem \ref{exiful.p2} -- local existence]
With the uniform bounds from Corollary \ref{bounds.p2}, it remains  to let $\epsilon$ to $0$ and to show the compactness of the set $\{(f_\epsilon,u_\epsilon)\}_{\epsilon>0}$ in appropriate spaces and that the limits of $(f_\epsilon, u_\epsilon)$ solve \eqref{sys} in the sense of Definition \ref{sol.p2}.

By virtue of the previously proved uniform bounds it follows that $u_\epsilon$ is uniformly bounded in ${\mathcal H}\hookrightarrow L^2(0,T;W^{2,2}(\tor^3))$ and
 $\partial_t u_\epsilon$ is uniformly bounded in $L^2(0,T;L^2(\tor^3))$. Since it holds 
$
W^{2,2}(\tor^3)\hookrightarrow\hookrightarrow W^{1,2}(\tor^3)\hookrightarrow L^2(\tor^3),
$
by Aubin--Lions lemma, we may extract from $u_\epsilon$ a strongly convergent subsequence in $L^2(0,T;W^{1.2}(\tor^3))$. Thus up to a subsequence
\begin{align}\label{ul2.p2}
u_\epsilon\to u \ \ \ {\rm in}\ L^2(0,T;L^2(\tor^3))
\end{align}
for some $u\in L^2(0,T;L^2(\tor^3))$ and
\begin{align}\label{nabu.p2}
\nablax  u_\epsilon \to \nablax  u \ \ \ {\rm in}\ L^2(0,T; L^2(\tor^3)).
\end{align}

On the other hand, the compactness of $f_\epsilon$ follows from bound the $L^\infty$ bound ($(iii)$ from
Proposition \ref{estimun.p2}) and Banach-Alaoglu Theorem. Then, up to a subsequence, $f_\epsilon\to f$ weakly * in $L^\infty([0,T]\times\tor^3\times\r^3)$. 

To finish the proof we need to show that $(f,u)$ satisfies $(\ref{regul.p2})$ in the sense of Definition \ref{sol.p2}. By (\ref{ul2.p2})
\begin{align}\label{dtu}
\int_0^T\int_{\tor^3} - u_\epsilon\cdot \partial_t\phi dxdvdt \to \int_0^T\int_{\tor^3} - u\cdot \partial_t\phi dxdvdt,
\end{align}
for all divergence free smooth $\phi$ with compact support in $t$. Thus $\partial_t u_\epsilon\to \partial_t u$ in the distributional sense, where $\partial_t u$ is the distributional derivative of $u$. However,
since $\partial_t u_\epsilon$ is bounded in $L^2(0,T;L^2(\tor^3))$, it  actually implies that $\partial_t u_\epsilon\to \partial_t u$ weakly in $L^2(0,T;L^2(\tor^3))$. 

By (\ref{ul2.p2}) and (\ref{nabu.p2}) $u_\epsilon\to u$ and $\nablax u_\epsilon\to\nablax u$ a.e. (up to a subsequence), which implies that also the convective term $(u_\epsilon\cdot\nablax)u_\epsilon\to(u\cdot\nablax)u$ a.e. Moreover, for
a sufficiently small $\eta>0$ we have
\begin{align*}
\|(u_\epsilon\cdot\nablax)u_\epsilon\|_{L^{1+\eta}(0,T;L^{1+\eta}(\tor^3))}\leq \|u\|_{L^\infty(0,T;L^2(\tor^3))}\|\nablax u\|_{L^p((0,T)\times \Omega)}\leq \|u\|_{\mathcal H}^2\ ,
\end{align*}
which means that $(u_\epsilon\cdot\nablax)u_\epsilon$ is uniformly bounded in $L^{1+\eta}$ and thus it is uniformly integrable. By Vitali's convergence theorem 
\begin{align}\label{convu}
\int_0^T\int_{\tor^3}(u_\epsilon\cdot\nablax)u_\epsilon\cdot \phi dxdvdt \to\int_0^T\int_{\tor^3}(u\cdot\nablax)u\cdot\phi dxdvdt,
\end{align}
for all divergence free smooth $\phi$ with compact support in $(0,T)$.

Similarly, up to a subsequence $\tau(Du_\epsilon)\to\tau(Du)$ a.e. and by (\ref{tau2.p2}) it is bounded in $L^{p^{'}}(0,T;L^{p^{'}}(\tor^3))$. Vitali's convergence theorem implies that $\tau(Du_\epsilon)\to\tau(Du)$ 
strongly in $L^{p^{'}-\eta}(0,T;L^{p^{'}-\eta})$ for some $\eta>0$. On the other hand, by Banach--Alaoglu Theorem the  sequence $\{\tau(Du_\epsilon)\}_{\epsilon>0}$ converges weakly in $L^{p^{'}}(0,T;L^{p^{'}}(\tor^3))$ to 
$\tau(Du)$ and by weak sequential lower semicontinuity of the norm $\tau(Du)\in L^{p^{'}}(0,T;L^{p^{'}}(\tor^3))$. Whence
\begin{align}\label{stressu}
\int_0^T\int_{\tor^3} \tau(Du_\epsilon):D\phi dxdvdt \to \int_0^T\int_{\tor^3} \tau(Du):D\phi dxdvdt
\end{align}
for all divergence free smooth $\phi\in C^\infty$ with compact support in $t$. Convergence and boundedness of the external force follows by similar arguments. Altogether, (\ref{dtu})--(\ref{stressu}) imply that for all 
divergence free smooth $\phi$ with compact support in the time variable we have
\begin{align}\label{slabo.p2}
\int_0^T\int_{\tor^3} \big(- u\cdot \partial_t\phi + (u\cdot\nablax)u\cdot \phi  + \tau(Du):D\phi\big) dxdvdt = -\int_0^T\int_{\tor^3}\int_{\r^3}(u-v)\cdot\phi fdxdvdt.
\end{align}

As $u$ is a limit of $u_\epsilon$ we also have
\begin{align*}
u\in L^\infty(0,T;\dot{W}^{1,p}(\tor^3))\cap L^\infty(0,T;W^{1,2}(\tor^3))\cap L^2(0,T;W^{2,2}(\tor^3))
\end{align*}
and since $\partial_t u\in L^2(0,T;L^2(\tor^3))$, the well-known result on the Gelfand triplet allows us to conclude that $u\in C([0,T];L^2(\tor^3))$ and thus $u\in{\mathcal H}$.

Finally, due to the sufficient regularity of $u$ we may replace (\ref{slabo.p2}) by equation from point 4 of Definition \ref{sol.p2} and by a density argument extend the class of admissible test functions
to $W^{1,2}(\tor^3)\cap \dot{W}^{1,p}_{div}(\tor^3)$.

As for the particle part of the solution $f$, thanks to the regularising effect of $\psi$ in $\FCS(f)$ and of sufficient regularity of $u$, converging with each term of the weak formulation for $f_\epsilon$ is straightforward.  
\end{proof}

\subsection{Global existence}\label{glob}

This part is dedicated to show that $T^*=T$. The previous subsection gave the local existence for the original problem. To show the existence on the whole time interval $[0,T]$
is sufficient to construct a priori estimate controlling traces in time which allow to prolong the solution till time $T$. As $\theta_\epsilon u_\epsilon \to u_\epsilon$ in the regularity class, where $u_\epsilon$ is bounded, it is not 
difficult to show that estimate \eqref{1} takes the form
\begin{align*}
\big(M_2f &+ \|u\|_2^2\big)(t) + c_1\kappa\int_0^t\|\nablax  u\|_p^p ds + \int_0^t\int_{\tor^3\times\r^3} |u - v|^2f dxdvds 
\leq \big(M_2f + \|u\|_2^2\big)(0).
\end{align*}
Note that the right-hand side depends only on the initial data of our problem. We were not able to use this argument previously and this led to the necessity to prove the local existence first.

Then in order to apply Proposition \ref{pomoc1.p2} we estimate the drag force
\begin{align*}
G=\int_{\r^3}(u-v)fdv
\end{align*}
in the same way as we estimated $G_\epsilon$ in the proof of Proposition \ref{momprop.p2}. So we find the better integrability of $u$. Hence for all $t\in [0,T^*)$ 
we are able to construct the a priori estimate without dependence from $T^*$, guaranteeing that, by continuity $u(T^*) \in W^{2,2}(\tor^3)$, $f(T^*)\in (L^1\cap L^\infty\cap W^{1,2})(\tor^3 \times \r^3)$. Hence we are able to 
prolong the solution over $T^*$. We proved $T^*=T$, hence the solution exists in fact in the whole time interval $(0,T)$.

%
%
%
%

%
%
%
%

\section{Large-time behavior of the solutions}
To prove Theorem \ref{t-a} we apply the strategy from \cite{bae2}, where the proof follows directly from the following lemma.
\begin{lem}\label{t-a-l}
Recall ${\mathcal E}$ is defined in \eqref{e}. Let $(f,u)$ be any solution of system \eqref{sys} in the sense of Definition \ref{sol.p2}. Then we have
\begin{enumerate}[(i)]
\item $\frac{d{\mathcal E}_p}{dt}\leq -2\psi(\sqrt{2}){\mathcal E}_p + 2\int_{\tor^3\times\r^3}(v-v_c)\cdot(u-v)fdxdv$,
\item $\frac{d{\mathcal E}_f}{dt}\leq -2c_1\kappa\varpi{\mathcal E}_f + 2\int_{\tor^3\times\r^3}(u_c-u)\cdot (u-v)fdxdv$,
\item $\frac{d{\mathcal E}_d}{dt}\leq -4\int_{\tor^3\times\r^3}(u_c-v_c)\cdot(u-v)fdxdv$.
\end{enumerate}
\end{lem}
\begin{proof}
The proof can be found in \cite{bae2} Lemma 4.1. The only slightly different part is the proof of $(ii)$ which we present below. We test $\eqref{sys}_2$ with $u-u_c$ (which is an admissible test function) obtaining
\begin{align*}
\int_{\tor^3}\partial_t u\cdot (u-u_c) dx + \int_{\tor^3}(u\cdot\nablax) u\cdot(u-u_c)dx + \int_{\tor^3}\tau(Du):Dudx = -\int_{\tor^3\times\r^3}(u-v)\cdot(u-u_c)fdxdv. 
\end{align*}
Since $\di u = 0$, the convective term dissapears, i.e.
\begin{align*}
\int_{\tor^3}(u\cdot\nablax) u\cdot(u-u_c)dx = 0.
\end{align*}
Moreover, by \eqref{tau2.p2} and  Korn's and Poincar\' e's inequalities we have
\begin{align*}
\int_{\tor^3}\tau(Du):Dudx\geq c_1\kappa\|\nablax u\|_2^2\geq c_1\kappa\varpi\int_{\tor^3}|u-u_c|^2dx = c_1\kappa\varpi {\mathcal E}_f.
\end{align*}
Finally we note that
\begin{align*}
\int_{\tor^3}\partial_t u\cdot (u-u_c) dx = \int_{\tor^3}\partial_t u\cdot (u-u_c) dx - \frac{d}{dt}u_c\cdot\int_{\tor^3}(u-u_c)dx = \frac{1}{2}\frac{d{\mathcal E}_f}{dt}
\end{align*}
and combine the above estimates to get
\begin{align*}
\frac{d{\mathcal E}}{dt}\leq -2c_1\kappa\varpi{\mathcal E}_f + 2\int_{\tor^3\times\r^3}(u-v)\cdot(u_c-u)fdxdv.
\end{align*}
\end{proof}

\begin{proof}[Proof of Theorem \ref{t-a}]
Fix $T>0$. Since 
$${\mathcal E} = 2{\mathcal E}_p + 2{\mathcal E}_f + {\mathcal E}_d,$$ it follows from Lemma \ref{t-a-l} that
\begin{align}\label{siic1}
\frac{d{\mathcal E}}{dt}\leq -4\psi(\sqrt{2}){\mathcal E}_p - 4c_1\kappa\varpi{\mathcal E}_f -4\int_{\tor^3\times\r^3}|u-v|^2fdxdv
\end{align}
and it proves that ${\mathcal E}$ is nonincreasing.
In order to prove \eqref{exp}, we aim to apply Gronwall's inequality to \eqref{siic1}, hence we need a term containing ${\mathcal E}_d$ on the right-hand side, which we extract from $-4\int_{\tor^3\times\r^3}|u-v|^2dxdv$. We have
\begin{align*}
\int_{\tor^3\times\r^3}|u-v|^2fdxdv = \int_{\tor^3\times\r^3}|u-u_c+u_c-v_c+v_c-v|^2fdxdv\\
=\int_{\tor^3}|u-u_c|^2m_0fdx + |u_c-v_c|^2 + \int_{\tor^3\times\r^3}|v-v_c|^2fdxdv\\
+ 2\int_{\tor^3\times\r^3}(u-u_c)\cdot(u_c-v_c)fdxdv + 2\int_{\tor^3\times\r^3}(u_c-v_c)\cdot(v_c-v)fdxdv\\ + 2\int_{\tor^3\times\r^3}(u-u_c)\cdot(v_c-v)fdxdv.
\end{align*}
The middle term on the right-hand side of the above equation is equal to $0$, i.e.
\begin{align*}
\int_{\tor^3\times\r^3}(u_c-v_c)\cdot(v_c-v)fdxdv = (u_c-v_c)\cdot\int_{\tor^3\times\r^3}(v_c-v)fdxdv = 0.
\end{align*}

The remaining terms can be estimated in the following way using Young's inequality
\begin{align*}
\left|2\int_{\tor^3\times\r^3}(u-u_c)\cdot(u_c-v)fdxdv\right|\leq \delta\int_{\tor^3}|u-u_c|^2m_0fdx + \frac{1}{\delta}\int_{\tor^3\times\r^3}|u_c-v|^2fdxdv\\
=\delta\int_{\tor^3}|u-u_c|^2m_0fdx + \frac{1}{\delta}\left(|u_c-v_c|^2 + \int_{\tor^3\times\r^3}|v_c-v|^2fdxdv\right)\\ = \delta\int_{\tor^3}|u-u_c|^2m_0fdx + \frac{1}{\delta}\left({\mathcal E}_d + {\mathcal E}_p\right).
\end{align*}

Altogether, we end up with the following estimate
\begin{align*}
\int_{\tor^3\times\r^3}|u-v|^2fdxdv\geq \left(1-\delta\right) \int_{\tor^3}m_0f|u-u_c|^2dx + \left(1-\frac{1}{\delta}\right)\left({\mathcal E}_d + {\mathcal E}_p\right),\quad \delta>0
\end{align*}
which together with \eqref{siic1} leads to
\begin{align*}
\frac{d{\mathcal E}}{dt} \leq -4\left(\psi(\sqrt{2})+1-\frac{1}{\delta}\right){\mathcal E}_p - 4c_1\kappa\varpi{\mathcal E}_f - 4\left(1-\frac{1}{\delta}\right){\mathcal E}_d\\
 - 4\left(1-\delta\right) \int_{\tor^3}m_0f|u-u_c|^2dx.
\end{align*}
Thus, by
\begin{align*}
\int_{\tor^3}m_0f|u-u_c|^2dx\leq \|m_0f\|_\infty{\mathcal E}_f,
\end{align*}
for $\delta=1+\eta$, we have

\begin{align}\label{siic2}
\frac{d{\mathcal E}}{dt} \leq -2\left(\psi(\sqrt{2})+1-\frac{1}{1+\eta}\right)\left(2{\mathcal E}_p\right) - 2\left(c_1\kappa\varpi-\eta\|m_0f\|_\infty\right){2\mathcal E}_f - 4\left(1-\frac{1}{1+\eta}\right){\mathcal E}_d.
\end{align}
Note that from the boundedness of the support in $v$ of $f$, for any $T>0$ we have $\sup_{t\leq T}\|m_0f\|_\infty<\infty$, thus after fixing
$$\eta:= \frac{c_1\kappa\varpi}{2\sup_{t\leq T}\|m_0f\|_\infty}>0$$
from \eqref{siic2} and Gronwall's inequality we deduce that
\begin{align}\label{argh}
{\mathcal E}(t)\leq {\mathcal E}(0)e^{-\gamma t}, \quad t\in[0,T),
\end{align}
where $\gamma:=\min\{2\psi(\sqrt{2})+\frac{2\eta}{1+\eta}, c_1\kappa\varpi, \frac{4\eta}{1+\eta}\}$. Finally, with an additional assumption that $\sup_{t\leq \infty}\|m_0f\|_\infty<\infty$ we can 
take $T=\infty$ in \eqref{argh} and in the definition of $\eta$.

In the case of $p>3$ one can use the advantage of the imbedding theorem. Esitmate (\ref{argh}) shows that $\mathcal{E}(t)$ is indeed a Lyapunov functional, it must decrease for all time. Now, taking the energy estimate \eqref{ee} we know that
\begin{equation}
 \int_0^T \Big( \int_{\tor^3} |\nablax u|^p dx + \int_{\tor^3}\int_{\r^3} |u-v|^2 f dxdv\Big)dt \leq bdd.
\end{equation}
with the  right-hand side independent of $T$. It means that there exists a sequence $t_n \to \infty$ (increasing) such that
\begin{equation}
 \int_{\tor^3} |\nablax u(t_n)|^p dx + \int_{\tor^3}\int_{\r^3} |u(t_n)-v|^2 f(t_n) dxdv \to 0 \mbox{ as } n \to \infty.
\end{equation}
The Sobolev inequality yields
\begin{equation}
 \|u(t_n)-u_c(t_n)\|_{L^\infty}\to 0 \mbox{ and as a consequence } \mathcal{E}_f(t_n) \to 0.
\end{equation}
Next, we note that
\begin{equation}
 \int_{\tor^3}\int_{\r^3} |v-u_c(t_n)|^2f(t_n)dxdv = \int_{\tor^3}\int_{\r^3} \Big((u_c^2 + v^2)f - 2 u_c\cdot v f\Big)(t_n)dxdv.
\end{equation}
But the conservation of momentum \eqref{cm} states that
\begin{equation}
 \int_{\tor^3} \big(u + \int_{\r^3} v f dv\big)dx =0
\end{equation}
for a.a. $t$. So
\begin{equation}
 \int_{\tor^3}\int_{\r^3} |v-u_c(t_n)|^2f(t_n)dxdv = \int_{\tor^3}\int_{\r^3} \Big((u_c^2 + v^2)f\Big)(t_n)dxdv + 2 u_c^2(t_n).
\end{equation}
It means that 
\begin{equation}
 u_c(t_n)\to 0 \quad \mbox{ and }\quad \int_{\tor^3} \int_{\r^3} v^2f(t_n) \to 0.
 \end{equation}
This implies that $v_c(t_n) \to 0$. Hence we proved that $\mathcal{E}(t_n)\to 0$, so by the monotonicity of $\mathcal{E}$ we get
\begin{equation}
 \mathcal{E}(t)\to 0
\end{equation}
for a.a. $t\to\infty$.

\end{proof}

{\bf Acknowledgments.} The work of the first author has been partly supported by Polish NCN grant No  2014/13/B/ST1/03094. The work of the second author was supported by the Polish NCN grant PRELUDIUM 2013/09/N/ST1/04113.  The work of the third author was supported by the Czech Science Foundation (GA\v{C}R; grant no. 16-03230S).
%
%
%
%

\appendix
\section{Appendix A}\label{appb}
We present the basic tools used throughout the paper.

We present two crucial lemmas from \cite{org}.

\begin{lem}\label{mom.p2}
Let $\beta>0$ and $g$ be a nonnegative function in $L^\infty([0,T]\times\tor^3 \times\r^3)$. The following estimate holds for any $\alpha<\beta$:
\begin{align*}
m_\alpha g(t,x)\leq \left(\frac{4}{3}\pi\|g(t,x,\cdot)\|_\infty + 1\right)m_\beta g(t,x)^\frac{\alpha+3}{\beta+3},
\end{align*}
for a.a. $(t,x)$.
\end{lem}

\begin{proof}
The proof can be found in \cite{org}, page 9 (Lemma 1).
\end{proof}

We include the formulation of the classical Gronwall's lemma  with it's less popular non-linear varieties.

\begin{lem}[Gronwall's lemma]\label{gron.p2}
Let $f$ be a nonnegative function satisfying inequality
\begin{align*}
f(t)\leq c + \int_{t_0}^t(a(s)f(s) + b(s)f^q(s))ds,\ c\geq 0,\ q>1,
\end{align*}
where $a$ and $b$ are nonnegative, integrable functions for $t\geq t_0$. Then we have
\begin{align*}
f(t)\leq c\left[e^{(1-q)\int_{t_0}^ta(s)ds}-c^{-1}(q-1)\int_{t_0}^tb(s)e^{(1-q)\int_s^ta(r)dr}ds\right]^\frac{1}{q-1}
\end{align*}
for $t\in[t_0,h]$ for $h>0$ provided that
\begin{align*}
c<\left[e^{(1-q)\int_{t_0}^{t_0+h}a(s)ds}\right]^\frac{1}{q-1}\left[(q-1)\int_{t_0}^{t_0+h}b(s)ds\right]^\frac{1}{1-q}.
\end{align*}
\end{lem}

\bibliographystyle{abbrv}
\bibliography{kupmpp}

\def\ocirc#1{\ifmmode\setbox0=\hbox{$#1$}\dimen0=\ht0 \advance\dimen0
  by1pt\rlap{\hbox to\wd0{\hss\raise\dimen0
  \hbox{\hskip.2em$\scriptscriptstyle\circ$}\hss}}#1\else {\accent"17 #1}\fi}
  \def\ocirc#1{\ifmmode\setbox0=\hbox{$#1$}\dimen0=\ht0 \advance\dimen0
  by1pt\rlap{\hbox to\wd0{\hss\raise\dimen0
  \hbox{\hskip.2em$\scriptscriptstyle\circ$}\hss}}#1\else {\accent"17 #1}\fi}
  \def\ocirc#1{\ifmmode\setbox0=\hbox{$#1$}\dimen0=\ht0 \advance\dimen0
  by1pt\rlap{\hbox to\wd0{\hss\raise\dimen0
  \hbox{\hskip.2em$\scriptscriptstyle\circ$}\hss}}#1\else {\accent"17 #1}\fi}
  \def\ocirc#1{\ifmmode\setbox0=\hbox{$#1$}\dimen0=\ht0 \advance\dimen0
  by1pt\rlap{\hbox to\wd0{\hss\raise\dimen0
  \hbox{\hskip.2em$\scriptscriptstyle\circ$}\hss}}#1\else {\accent"17 #1}\fi}
\begin{thebibliography}{10}

\bibitem{ahn1}
S.~M. Ahn, H.~Choi, S.-Y. Ha, and H.~Lee.
\newblock On collision-avoiding initial configurations to {C}ucker-{S}male type
  flocking models.
\newblock {\em Commun. Math. Sci.}, 10(2):625--643, 2012.

\bibitem{Bae}
H.-O. Bae, Y.-P. Choi, S.-Y. Ha, and M.-J. Kang.
\newblock Time-asymptotic interaction of flocking particles and an
  incompressible viscous fluid.
\newblock {\em Nonlinearity}, 25(4):1155--1177, 2012.

\bibitem{bae2}
H.-O. Bae, Y.-P. Choi, S.-Y. Ha, and M.-J. Kang.
\newblock Global existence of strong solution for the
  {C}ucker-{S}male-{N}avier-{S}tokes system.
\newblock {\em J. Differential Equations}, 257(6):2225--2255, 2014.

\bibitem{org}
L.~Boudin, L.~Desvillettes, C.~Grandmont, and A.~Moussa.
\newblock Global existence of solutions for the coupled {V}lasov and
  {N}avier-{S}tokes equations.
\newblock {\em Differential Integral Equations}, 22(11-12):1247--1271, 2009.

\bibitem{rec}
J.~A. Carrillo, Y.-P. Choi, and M.~Hauray.
\newblock The derivation of swarming models: mean-field limit and {W}asserstein
  distances.
\newblock In {\em Collective dynamics from bacteria to crowds}, volume 553 of
  {\em CISM Courses and Lect.}, pages 1--46. Springer, Vienna, 2014.

\bibitem{ccmp}
J.~A. Carrillo, Y.-P. Choi, P.~B. Mucha, and J.~Peszek.
\newblock {S}harp conditions to avoid collisions in singular {C}ucker-{S}male
  interactions.
\newblock {\em preprint}, arXiv:1609.03447v1, 2016.

\bibitem{car}
J.~A. Carrillo, M.~Fornasier, J.~Rosado, and G.~Toscani.
\newblock Asymptotic flocking dynamics for the kinetic {C}ucker-{S}male model.
\newblock {\em SIAM J. Math. Anal.}, 42(1):218--236, 2010.

\bibitem{car3}
J.~A. Carrillo, A.~Klar, S.~Martin, and S.~Tiwari.
\newblock Self-propelled interacting particle systems with roosting force.
\newblock {\em Math. Models Methods Appl. Sci.}, 20(suppl. 1):1533--1552, 2010.

\bibitem{choi}
Y.-P. Choi.
\newblock Large-time behavior for the {V}lasov/compressible {N}avier-{S}tokes
  equations.
\newblock {\em J. Math. Phys.}, 57(7):071501, 13, 2016.

\bibitem{CoSe}
P.~Constantin and G.~Seregin.
\newblock Global regularity of solutions of coupled {N}avier-{S}tokes equations
  and nonlinear {F}okker {P}lanck equations.
\newblock {\em Discrete Contin. Dyn. Syst.}, 26(4):1185--1196, 2010.

\bibitem{cuc2}
F.~Cucker and J.-G. Dong.
\newblock Avoiding collisions in flocks.
\newblock {\em IEEE Trans. Automat. Control}, 55(5):1238--1243, 2010.

\bibitem{cuc3}
F.~Cucker and C.~Huepe.
\newblock Flocking with informed agents.
\newblock {\em MathS in Action}, 1(1):1--25, 2008.

\bibitem{cuc4}
F.~Cucker and E.~Mordecki.
\newblock Flocking in noisy environments.
\newblock {\em J. Math. Pures Appl. (9)}, 89(3):278--296, 2008.

\bibitem{cuc1}
F.~Cucker and S.~Smale.
\newblock Emergent behavior in flocks.
\newblock {\em IEEE Trans. Automat. Control}, 52(5):852--862, 2007.

\bibitem{deg1}
P.~Degond and S.~Motsch.
\newblock Macroscopic limit of self-driven particles with orientation
  interaction.
\newblock {\em C. R. Math. Acad. Sci. Paris}, 345(10):555--560, 2007.

\bibitem{deg2}
P.~Degond and S.~Motsch.
\newblock Large scale dynamics of the persistent turning walker model of fish
  behavior.
\newblock {\em J. Stat. Phys.}, 131(6):989--1021, 2008.

\bibitem{dua1}
R.~Duan, M.~Fornasier, and G.~Toscani.
\newblock A kinetic flocking model with diffusion.
\newblock {\em Comm. Math. Phys.}, 300(1):95--145, 2010.

\bibitem{has}
R.~Erban, J.~Ha\v{s}kovec, and Y.~Sun.
\newblock A {C}ucker-{S}male model with noise and delay.
\newblock {\em SIAM J. Appl. Math.}, 76(4):1535--1557, 2016.

\bibitem{go1}
T.~Goudon, L.~He, A.~Moussa, and P.~Zhang.
\newblock The {N}avier-{S}tokes-{V}lasov-{F}okker-{P}lanck system near
  equilibrium.
\newblock {\em SIAM J. Math. Anal.}, 42(5):2177--2202, 2010.

\bibitem{go2}
T.~Goudon, P.-E. Jabin, and A.~Vasseur.
\newblock Hydrodynamic limit for the {V}lasov-{N}avier-{S}tokes equations. {I}.
  {L}ight particles regime.
\newblock {\em Indiana Univ. Math. J.}, 53(6):1495--1515, 2004.

\bibitem{go3}
T.~Goudon, P.-E. Jabin, and A.~Vasseur.
\newblock Hydrodynamic limit for the {V}lasov-{N}avier-{S}tokes equations.
  {II}. {F}ine particles regime.
\newblock {\em Indiana Univ. Math. J.}, 53(6):1517--1536, 2004.

\bibitem{hahaki}
S.-Y. Ha, T.~Ha, and J.-H. Kim.
\newblock Asymptotic dynamics for the {C}ucker-{S}male-type model with the
  {R}ayleigh friction.
\newblock {\em J. Phys. A}, 43(31):315201, 19, 2010.

\bibitem{hajekaka}
S.-Y. Ha, E.~Jeong, J.-H. Kang, and K.~Kang.
\newblock Emergence of multi-cluster configurations from attractive and
  repulsive interactions.
\newblock {\em Math. Models Methods Appl. Sci.}, 22(8):1250013, 42, 2012.

\bibitem{hakalaru}
S.-Y. Ha, M.-J. Kang, C.~Lattanzio, and B.~Rubino.
\newblock A class of interacting particle systems on the infinite cylinder with
  flocking phenomena.
\newblock {\em Math. Models Methods Appl. Sci.}, 22(7):1250008, 25, 2012.

\bibitem{ha}
S.-Y. Ha, D.~Ko, Y.~Zhang, and X.~Zhang.
\newblock Emergent dynamics in the interactions of {C}ucker-{S}male ensembles.
\newblock {\em Kinet. Relat. Models}, 10(3):689--723, 2017.

\bibitem{halele}
S.-Y. Ha, K.~Lee, and D.~Levy.
\newblock Emergence of time-asymptotic flocking in a stochastic
  {C}ucker-{S}male system.
\newblock {\em Commun. Math. Sci.}, 7(2):453--469, 2009.

\bibitem{haliu}
S.-Y. Ha and J.-G. Liu.
\newblock A simple proof of the {C}ucker-{S}male flocking dynamics and
  mean-field limit.
\newblock {\em Commun. Math. Sci.}, 7(2):297--325, 2009.

\bibitem{hatad}
S.-Y. Ha and E.~Tadmor.
\newblock From particle to kinetic and hydrodynamic descriptions of flocking.
\newblock {\em Kinet. Relat. Models}, 1(3):415--435, 2008.

\bibitem{hamd}
K.~Hamdache.
\newblock Global existence and large time behaviour of solutions for the
  {V}lasov-{S}tokes equations.
\newblock {\em Japan J. Indust. Appl. Math.}, 15(1):51--74, 1998.

\bibitem{mal}
J.~M{\'a}lek, J.~Ne{\v{c}}as, M.~Rokyta, and M.~R{\ocirc{u}}{\v{z}}i{\v{c}}ka.
\newblock {\em Weak and measure-valued solutions to evolutionary {PDE}s},
  volume~13 of {\em Applied Mathematics and Mathematical Computation}.
\newblock Chapman \& Hall, London, 1996.

\bibitem{mo}
S.~Motsch and E.~Tadmor.
\newblock A new model for self-organized dynamics and its flocking behavior.
\newblock {\em J. Stat. Phys.}, 144(5):923--947, 2011.

\bibitem{mp}
P.~B. Mucha and J.~Peszek.
\newblock The cucker-smale equation: singular communication weight, measure
  solutions and weak-atomic uniqueness.
\newblock {\em preprint}, arXiv:1509.07673v1, 2015.

\bibitem{jpe}
J.~Peszek.
\newblock Existence of piecewise weak solutions of a discrete
  {C}ucker--{S}male's flocking model with a singular communication weight.
\newblock {\em J. Differential Equations}, 257(8):2900--2925, 2014.

\bibitem{jps}
J.~Peszek.
\newblock Discrete {C}ucker-{S}male flocking model with a weakly singular
  weight.
\newblock {\em SIAM J. Math. Anal.}, 47(5):3671--3686, 2015.

\bibitem{poko}
M.~Pokorn{\'y}.
\newblock Cauchy problem for the non-{N}ewtonian viscous incompressible fluid.
\newblock {\em Appl. Math.}, 41(3):169--201, 1996.

\bibitem{shen}
J.~Shen.
\newblock Cucker-{S}male flocking under hierarchical leadership.
\newblock {\em SIAM J. Appl. Math.}, 68(3):694--719, 2007/08.

\bibitem{tontu}
J.~Toner and Y.~Tu.
\newblock Flocks, herds, and schools: a quantitative theory of flocking.
\newblock {\em Phys. Rev. E (3)}, 58(4):4828--4858, 1998.

\bibitem{top}
C.~M. Topaz and A.~L. Bertozzi.
\newblock Swarming patterns in a two-dimensional kinematic model for biological
  groups.
\newblock {\em SIAM J. Appl. Math.}, 65(1):152--174, 2004.

\end{thebibliography}
\end{document}